\documentclass[12pt]{article}
\usepackage{amsmath,amsthm,amsfonts,amssymb,mathrsfs,bm}
\usepackage[usenames]{color}
\usepackage{hyperref}
\usepackage{amssymb}

\parskip 	\bigskipamount

\newtheorem{theorem}{Theorem}[section]

\newtheorem{corollary}[theorem]{Corollary}
\newtheorem{proposition}[theorem]{Proposition}

\newtheorem{remark}{Remark}[section]

\newtheorem{definition}{Definition}

\newtheorem{question}{Question}

\def\Z{\mathbb{Z}}
\def\R{\mathbb{R}}
\def\C{\mathbb{C}}
\def\N{\mathbb{N}}
\def\E{\mathbb{E}}
\def\P{\mathbb{P}}
\def\pf{\mathbb{P}^f}
\def\td{\mathbb{T}^d}
\def\t{\mathbb{T}}
\def\wt{\widetilde}
\def\ro{\rho}
\def\o{\omega}
\def\ou{\omega_{\mathrm{out}}}
\def\oin{\omega_{\mathrm{in}}}
\def\ze{x_0}
\def\ph{\varphi}
\def\var{\mathrm{Var}}
\def\sin{\mathrm{sin}}
\def\leb{\mathcal{L}}
\def\L{\Lambda}
\def\la{\lambda}
\def\eps{\epsilon}
\def\del{\delta}
\def\hf{\hat{f}}
\def\H{\mathcal{H}}
\def\u{\Upsilon}

\def\F{\mathfrak{F}}
\def\e{\mathcal{E}}
\def\G{\mathcal{G}}
\def\arg{\text{Arg}}
\def\hm{\text{HM}}
\def\el{\mathcal{L}}

\def\ol{\overline}
\def\q{\mathcal{Q}}
\def\U{\u_R}
\renewcommand{\l}[0]{\left }
\renewcommand{\r}[0]{\right}

\makeatletter
\renewcommand*{\@cite@ofmt}{\hbox}
\makeatother

\begin{document}
\title{Determinantal processes and completeness of random exponentials: the critical case}
\author{\href{httpRevision://www.princeton.edu/~sg18/}{Subhroshekhar Ghosh}\\ Princeton University \\ sg18@princeton.edu}
\date{}
\maketitle

\begin{abstract}
 For a locally finite point set $\L \subset \R$, consider the collection of exponential functions given by $\e_{\L}:=\{ e^{i \la x } : \la \in \L\}$. We examine the question whether $\e_{\L}$ spans the Hilbert space $L^2[-\pi,\pi]$, when $\L$ is random. For several point processes of interest, this belongs to a certain critical case of the corresponding question for deterministic $\L$, about which little is known. For $\L$ the continuum sine kernel process, obtained as the bulk limit of GUE eigenvalues, we establish that $\e_{\L}$ is indeed complete almost surely. We also answer an analogous question on $\C$ for the Ginibre ensemble, arising as weak limits  of the spectra of certain non-Hermitian Gaussian random matrices. In fact we establish completeness for any ``rigid'' determinantal point process in a general setting. In addition, we partially answer two questions of Lyons and Steif about stationary determinantal processes on $\Z^d$. 
\end{abstract}

\newpage
\section{Introduction}

Any locally finite point set $\L \subset \R$ gives us a set of functions \[\e_{\L}:=\{ e_{\la} :{\la \in \L}\} \subset L^2[-\pi,\pi],\] where $e_{\la}(x) = e^{i\la x}$,  $i$ being the imaginary unit.  A set of vectors is said to ``span'' a Hilbert space if their closed linear span equals the Hilbert space under consideration.  
The following question is classical:
\begin{question}
\label{q0}
Does  $\e_{\L}$ span $L^2[-\pi,\pi]$ ?
\end{question}

An equivalent terminology found in the literature to describe the fact that $\e_{\L}$ spans $L^2[-\pi,\pi]$ is that $\e_{\L}$ is complete in $L^2[-\pi,\pi]$.
When $\L$ is deterministic, this a well studied problem in the literature. In the case where $\L$ is random, that is, $\L$ is a point process, the literature is far more limited (other than what can be deduced from the results in the deterministic setting). 
% Of course, Question \ref{q4} can be asked for any point process of unit intensity in $\R$. To the best of our knowledge, the answer is unknown even for the  standard Poisson process on $\R$ of unit intensity. 
For  any ergodic point process $\L$, it can be easily checked that the event in question has a 0-1 law. 

In this paper we provide a complete answer to Question \ref{q0} in the case where $\L$ is the continuum sine kernel process (see Section \ref{defs} for a precise definition):
\begin{theorem}
 \label{sin}
When $\L$ is a realisation of the continuum sine kernel process on $\R$, almost surely $\e_{\L}$ spans $L^2[-\pi,\pi]$. 
\end{theorem}

Similar questions can be asked in higher dimensions as well. On $\C$, we consider the analogous question with $\L$ coming from the  Ginibre ensemble (see Section \ref{defs} for a precise definition). An exponential function here is defined as $e_{\la}(z)= e^{\overline{\la}z}$ and the natural space in which to study completeness is the Fock-Bargmann space. The latter space is the closure of the set of polynomials (in one complex variable) in $L^2\l(\gamma\r)$ where $\gamma$ is the standard complex Gaussian measure on $\C$, having the density $\frac{1}{\pi}e^{-|z|^2}$ with respect to the Lebesgue measure. In this case we prove:
\begin{theorem}
\label{gin}
When $\L$ is a realisation of the Ginibre ensemble on $\C$, a.s. $\e_{\L}$ spans the Fock-Bargmann space. Equivalently, a.s. in $\L$ the following happens: if there is a function $f$ in the Fock-Bargmann space which vanishes at all the points of $\L$, then $f\equiv 0$.
\end{theorem}

All these questions are specific realisations of the following completeness question that was asked of any determinantal process  by Lyons and Peres. 

Consider a determinantal point process $\Pi$ in a space $\Xi$ equipped with a background measure $\mu$. Let the point process $\Pi$ correspond to a projection onto the subspace $\H$ of $L^2(\mu)$ in the usual way; for details, see Section \ref{defs} and also \cite{HKPV}, \cite{Sos} and \cite{Ly}. Let $K(\cdot,\cdot)$ be the kernel of the determinantal process, which is also the integral kernel corresponding to the projection onto $\H$. Consequently, $\H$ is a reproducing kernel Hilbert space, with the kernel $K(\cdot,\cdot)$.  Let $\{x_i\}_{i=1}^{\infty}$ be a sample from $\Pi$. Clearly, $\{K(\cdot,x_i)\}_{i=1}^{\infty} \subset \H$. Lyons and Peres asked  the  following question for any arbitrary determinantal point process associated with a projection kernel (see, e.g., \cite{Ly1} Section 4 and  Conjecture 4.6 therein):

\begin{question}
 \label{q1}
Is the random set of functions $\{K(\cdot,x_i)\}_{i=1}^{\infty}$ complete in $\H$ a.s.?
\end{question}
Results of a similar nature have been obtained in special cases. The answer to Question \ref{q1} is trivial in the case where $\H$ is finite dimensional (say the dimension is $N$), there it follows simply from the fact that the matrix $(K(x_i,x_j))_{i,j=1}^N$ is a.s. non-singular on one hand, and it is the Gram matrix of the vectors $\{K(\cdot,x_i)\} \subset \H$ on the other. In the case where $\Xi$ is a countable space, this was first proved, albeit implicitly, for the case of spanning forests by Morris \cite{Mo}, although his theorem was not stated in  terms of the completeness problem that we have been discussing. For a detailed explanation of the connection between Morris' theorem and the completeness problem, we refer the reader to \cite{Ly}, in particular Section 7 and the discussion  preceding Theorem 7.2 therein.  Subsequently,  this matter has been settled in the affirmative for any discrete determinantal process by Lyons, see \cite{Ly}. However, in the continuum (e.g. when $\Xi=\R^d$ and $\H$ is infinite dimensional), the answer to Question \ref{q1} is unknown.

In this paper, we answer Question \ref{q1} in the affirmative for rigid determinantal processes. Recall that a locally compact metric space is called \textit{proper} if every closed ball (of finite radius) in that space is compact.

\begin{theorem}
\label{rigcomp}
Let $\Pi$ be a determinantal point process with a kernel $K(\cdot,\cdot)$ on a second countable locally compact  metric space $E$ (equipped with a proper metric $d$) and a background measure $\mu$ which is a non-negative regular Borel measure. Suppose  
% that is diffuse, in the sense that for each  $x \in E$ and $r \ge 0$, we have $\mu\l(\{y:d(x,y)=r\}\r)=0$. Suppose that  
% with a background measure $\mu$ that is mutually absolutely continuous with respect to the Lebesgue measure on $\R^d$, 
 $K(\cdot,\cdot)$,  as an integral operator from $L^2(\mu)$ to itself, is the projection onto a closed subspace $\H \subset L^2(\mu)$
% , and the map $x \mapsto K(\cdot,x)$ is continuous on the support of $\mu$, as a map from $E$ to $\H$.

Let $\Pi$ be rigid, in the sense that for any open ball $B$ with a finite radius, the point configuration outside $B$ a.s. determines the number of points $N_B$ of $\Pi$ inside $B$. 
Then $\{K(\cdot,x): x \in \Pi \}$ is a.s. complete in $\H$, that is, a.s. this set of functions spans $\H$.
\end{theorem}

A typical example of the setup described in Theorem \ref{rigcomp} is a rigid determinantal point process on a Euclidean space with a continuous kernel and a background measure that is mutually absolutely continuous with respect to Lebesgue measure. Many natural examples of determinanal point processes in the continuum, including the sine kernel process or the Ginibre ensemble, fit the above description.

To see the correspondence between Question \ref{q1} and Theorems \ref{sin} - \ref{gin}, we can make appropriate substitutions for the kernels and spaces in Theorem \ref{rigcomp}, for details see Section \ref{defs}. Briefly, the statement of Theorem \ref{sin} is equivalent to Question \ref{q1} (with an affirmative answer) under Fourier conjugation. The statement in Theorem \ref{gin} involving the vanishing of functions on $\L$ is a result of the fact that the Fock-Bargmann space is a reproducing kernel Hilbert space, with the reproducing kernel and background measure being the same as the determinantal kernel and the background measure of the Ginibre ensemble.

Completeness (in the appropriate Hilbert space) of collections of exponential functions indexed by a point configuration is a well-studied theme, for a classic reference see the survey by Redheffer \cite{Re}. However, most of the classical results deal with deterministic point configurations, and are often stated in terms of some sort of density of the underlying point set.  E.g., one crucial parameter is the Beurling Malliavin density of the point configuration, for details, see \cite{Ly} Definition 7.13 and the ensuing discussion there. Typically, the results are of the following form : if the relevant density parameter is supercritical, then the exponential system is complete, and if it is subcritical, it is incomplete. E.g., see Beurling and Malliavin's theorem, stated as Theorem 71 in \cite{Re}. 

However, it turns out that in our cases of interest, e.g. as in Theorems \ref{sin} and \ref{gin}, the density of the point process is almost surely equal to the critical density in terms of the classical results (see, e.g., Section 4.2 points 1 and 2 in \cite{Ly1}).  In this paper we have chosen the normalizations for our models such that the one-point intensity (equivalently, the critical density)  is equal to 1. 

The critical cases in the deterministic setting are more difficult to handle. E.g., in $L^2[-\pi,\pi]$, $\{e_{\la}:\la \in \Z\}$ is a complete set of exponentials, but $\{e_{\la}:\la \in \Z \setminus \{0\}\}$ is incomplete.  In the random case, when the densities are super or subcritical,  we can either invoke the results in the deterministic setting (e.g., Theorem 71 in \cite{Re}), or there is existing literature (see \cite{CLP} Theorems 1.1 and 1.2 or \cite{SU}). However, at critical densities, which is the case we are interested in, much less is known. To the best of our knowledge, the only known case is that of a perturbed lattice, where completeness was established under some regularity conditions on the (random) perturbations, see \cite{CL} Theorem 5. Our result in Theorem \ref{rigcomp} answers this question for natural point processes, like the Ginibre or the sine kernel, which are not independent perturbations of  a lattice.

There are other natural examples of determinantal point processes for which similar questions are not amenable to our approach. E.g., one can consider Question \ref{q1} for the zero process of the hyperbolic Gaussian analytic function, where the answer, either way, is unknown, because this determinantal point process is not rigid (see \cite{HoSo}).

En route proving Theorem \ref{rigcomp}, we establish a version of the negative association property for determinantal point processes in the continuum, which is relevant for our purposes.  In the discrete setting, a complete theorem to this effect was has been proved in \cite{Ly}. However,  we could not locate such a result in the literature on determinantal point processes in the continuum. In Theorem \ref{cna}, we establish negative association for the number of points for determinantal point processes in a general setting.
\begin{definition}
\label{negcordef}
 We call two non-negative real valued random variables $X$ and $Y$ negatively associated if for any real numbers $r$ and $s$ we have \[ \P \big(   \l( X > r \r) \cap \l( Y > s \r)  \big)  \le \P \big( X > r \big) P \big( Y > s \big). \] 
This is equivalent to the complementary condition 
\[  \P \big(   \l( X \le r \r) \cap \l( Y \le s \r)  \big)  \le \P \big( X \le r \big)\P \big( Y \le s \big).  \]
\end{definition}
% In the theory of determinantal point processes on $\R^d$, by a standard kernel we mean a Hermitian kernel $\k(x,y)$ which is continuous as a function from $\R^d \times \R^d \to \C$ and is a non-negative trace class contraction when viewed as an integral operator from $L^2(\mu)\to L^2(\mu)$, where $\mu$ is the background measure. For further details, we refer the interested reader to \cite{HKPV}. 

\begin{theorem}
\label{cna} 
% For any determinantal point process with a standard kernel on $\R^d$ and the background measure $\mu$ absolutely continuous with respect to the Lebesgue measure, the numbers of points in two disjoint Borel sets are negatively associated random variables, in the sense of Definition \ref{negcordef}. 
 Let $E$ be a second countable locally compact Hausdorff space, equipped with a non-negative regular Borel measure $\mu$. 
%  Let $\mu$ be also diffuse, in the sense that for every $x\in E$ and $r\ge 0$, we have $\mu\l(\{y:d(x,y)=r\}\r)=0$. 
 Let $\Pi$ be a determinantal point process on $E$ with a kernel $K$ and background measure $\mu$. Then, for every finite collection of disjoint  Borel sets $\{A_i\}_{i=1}^n \subset E$ and integers $\{m_i\}_{i=1}^n$, we have 
 \begin{equation} \label{neg0}\P\l(\bigcap_{i=1}^n \{N(A_i) \le m_i\}\r) \le \prod_{i=1}^n \P\l(N(A_i)\le m_i\r).\end{equation} 
 In particular, the random variables $N(A)$ and $N(B)$ are negatively associated for any two disjoint Borel sets $A$ and $B$.
\end{theorem}

Consequent to Theorem \ref{cna}, we also discuss briefly how the above notion of negative dependence implies more general notions of negative association for determinantal point processes.

In addition to the completeness questions for random exponentials defined with respect to natural determinantal processes, we also partially answer two questions asked by Lyons and Steif in \cite{LySt}. First, we give a little background on a certain class of stationary determinantal processes on $\Z^d$ studied in \cite{LySt}.

Let $f$ be a function $\td \to [0,1]$, where $\td=[0,1]^d$. Then multiplication by $f$ is a non-negative contraction operator from $L^2(\td)$ to itself. Under Fourier conjugation, this gives rise to a non-negative contraction operator   $Q: \ell^2(\Z^d) \to \ell^2(\Z^d)$. This, in turn, gives rise to a determinantal point process $\pf$ on $\Z^d$ in a canonical way, for details see \cite{Ly}. For a point configuration $\o$ on $\Z^d$, we denote by $\o(\underline{k})$ the indicator function of having a point at $\underline{k}\in \Z^d$ in the configuration $\o$. We denote by $\ou$ the configuration  of points on $\Z^d \setminus \mathbf{0}$ obtained by restricting $\o$ to $\Z^d \setminus \mathbf{0}$, where $\mathbf{0}$ denotes the origin in $\Z^d$.  

For a point process $\Pi$ on a space $\Xi$, we denote by $[\Pi]$ the (random) counting measure obtained from a realisation of $\Pi$. The $k$-point intensity functions of a point process, when they exist, will be denoted by $\rho_k, k \ge 1$. For the processes $\pf$, all intensity 
functions exist. Moreover, the translation invariance of $\pf$ implies that $\rho_k(x_1+x,\cdots,x_k+x)=\rho_k(x_1,\cdots,x_k)$ for all $x,x_1,\cdots,x_k \in \Z^d$, in particular, $\rho_1$ is a constant  $ \in [0,1]$.

In the paper \cite{LySt} it was conjectured (Conjecture 9.9 therein) that all determinantal processes obtained in this way are insertion and deletion tolerant, meaning that both $\P[\o(\mathbf{0})=1 | \ou ]>0$  and $\P[\o(\mathbf{0})=0 | \ou ]>0$ a.s. We resolve this conjecture in the negative, showing that for $f$ which is the indicator function of an interval in $\t$, this is not true. 

\begin{theorem}
\label{indic}
Let $f$ be the indicator function of an interval $I \subset \t$. Then there exists a measurable function \[N: \mbox {Point configurations on }\Z \setminus \mathbf{0} \to \N \cup \{0\} \] such that a.s. we have $\o(\mathbf{0})=N(\ou)$.  Consequently, the events  $\{\P[\o(\mathbf{0})=1 | \ou ]=0\}$ and  $\{\P[\o(\mathbf{0})=0 | \ou ]=0\}$ both have positive probability (in $\ou$).  
\end{theorem}

We end by partially answering another question from \cite{LySt}, where we demonstrate that ``almost all'' functions $f$ can be reconstructed from the distribution $\pf$.
\begin{theorem}
\label{deter}
Define $\e$ to be the set of functions  \[\e:=\{f \in L^{\infty}(\t): 0 \le f(x) \le 1 \text{ for almost every } x \in \t \}.\]
Then $\pf$ determines $f$ up to translation and flip, except possibly for a meagre set of functions in the $L^{\infty}$ topology on $\e$.
\end{theorem}
For any (as opposed to ``almost all'') function $f$, we prove that $\pf$ determines the value distribution of $f$.

\begin{proposition}
 \label{valdist}
For any $f \in \e$, $\pf$ determines the value distribution of $f$. This is true for $\pf$ defined on $\Z^d$ for any $d \ge 1$. 
\end{proposition}

\section{Definitions}
\label{defs}
In this section, we give precise descriptions of the models under study and discuss how Theorem \ref{rigcomp} is related to Theorems \ref{sin} and \ref{gin}.

A determinantal point process on a space $\Xi$ with background measure $\mu$ and kernel $K: \Xi \times \Xi \to \C$, is a point process whose $n$-point intensity functions (with respect to the measure $\mu^{\otimes n}$) are given by \[ \rho_n(x_1,\cdots,x_n)=\text{ det } \l(K(x_i,x_j)_{i,j=1}^n  \r) .\] 
The kernel $K$ induces an integral operator on $L^2(\mu)$, which must be locally trace class and, additionally, a non-negative contraction. An interesting class of examples is obtained when the integral operator given by $K$ is a projection onto a subspace $\H$ of $L^2(\mu)$.
 
Our first example is the Ginibre ensemble, which is obtained as above with $K(z,w)=e^{z\overline{w}}, d\mu(z)=\frac{1}{\pi}e^{-|z|^2}d\leb(z)$ and $\H$ is the Fock-Bargmann space $\subset L^2(\mu)$ (here $\leb$ is the Lebesgue measure on $\C$). For every $n$, we can consider an $n \times n$ matrix of i.i.d. complex Gaussian entries. The Ginibre ensemble arises as the weak limit (as $n \to \infty$) of the point process given by the eigenvalues of this matrix. The connection between Theorems \ref{rigcomp} and \ref{gin} is fairly straightforward. It has been proved in \cite{GNPS} Theorem 1.1 that the Ginibre ensemble is rigid in the sense of Theorem \ref{rigcomp}. Setting $\Pi=$ the Ginibre ensemble in Theorem \ref{rigcomp}, we deduce that the set  of functions $\{e_{\la}(z):=e^{\ol{\la} z} | \la \in \Pi  \}$ is a.s. complete in the Fock-Bargmann space. But the Fock-Bargmann space is a reproducing kernel Hilbert space with the kernel $K(z,w)=e^{\ol{w}z}$ and the standard complex Gaussian measure as the background measure. Therefore, for a function $f$ in the Fock-Bargmann space, orthogonality of $f$ to $K(\cdot,\la)$ is equivalent to $f(\la)=0$. So the statement that $f(\la)=0 \forall \la \in \Pi$ is equivalent to orthogonality of $f$ to the span of $\{K(\cdot,\la):\la \in \Pi\}$. Therefore, Theorem \ref{rigcomp}  implies that $f$ must be identically 0.

The continuum sine kernel process is given by $K(x,y)=\frac{\sin\pi(x-y)}{\pi(x-y)}, \mu = $ the Lebesgue measure on $\R$, $\H=$ the Fourier transforms of the set of $L^2$ functions supported on $[-\pi,\pi]$ (considered as a subspace of $L^2(\R)$). Here we define the Fourier transform of a Schwarz function $f$ to be $\hat{f}(\xi)= \frac{1}{2\pi} \int_{-\infty}^{\infty} f(x) e^{-ix\xi} dx $, and extend the definition to all $L^2$ functions $f$ by a standard density argument. The continuum sine kernel process arises as the bulk limit of the eigenvalues of the Gaussian Unitary Ensemble (GUE). Setting $f= 1_{[-\pi,\pi]}$ in Theorem \ref{contrig}, we obtain that the continuum sine kernel process is rigid in the sense of Theorem \ref{rigcomp}. Therefore, setting $\Pi$ to be the continuum sine kernel process in Theorem \ref{rigcomp}, we deduce that the functions $h_{\la}(t):=\frac{\sin\pi(\la -t)}{\pi(\la -t )}$, where $\la$ ranges over the points in $\Pi$, are a.s. complete in the Hilbert space $\H=$ the Fourier transforms of the set of $L^2$ functions supported on $[-\pi,\pi]$. Therefore, $\{\breve{h}_{\la}: \la \in \Pi\}$ is complete in $L^2(-\pi,\pi)$, where $\breve{h}$ denotes the inverse Fourier transform. But $\breve{h}_{\la}(x)=e^{i\la x}1_{[-\pi,\pi]}(x)=e_{\la}(x)$. Putting all these together, we obtain the fact that Theorem \ref{rigcomp} implies Theorem \ref{sin}.

For greater details on these point processes, see \cite{HKPV}.

\section{Completeness of random function spaces}

In this Section, we prove Theorem \ref{rigcomp}. Due to the connections discussed in Section \ref{defs}, this will automatically establish Theorems \ref{sin} and \ref{gin}. 
% On the way, we provide a proof of Theorem \ref{cna}.

Before the main theorem, we establish a preparatory result. 

\begin{proposition}
 \label{regularity}
Suppose $\Pi$ is a rigid point process on a locally compact metric space $(E,d)$, in the sense that for any open ball $B$ with  finite radius, the point configuration outside $B$ a.s. determines the number of points $N_B$ of $\Pi$ inside $B$. Assume that for any ball $B$, $\E\l[ N_B \r]<\infty$. Let $A(r)$ denote the closed annulus of thickness $r$ around $B$, and let $\Pi_{A(r)}$ denote the point configuration in $A(r)$ obtained by restricting $\Pi$ to $A(r)$. Then we have \begin{equation} \label{regu} 
\E\l[ N_B \big| \Pi_{A(r)} \r] \to N_B 
\end{equation}
a.s. as $r \to \infty$.
\end{proposition}
\begin{proof}
 This follows from Levy's 0-1 law and the convergence of the Doob's martingales $ M_r := \E \l[ N_B| \Pi_{A(r)} \r]$ , ${r \ge 0}$ of $N_B$, as $r \to \infty$. Note that $M_{\infty}=N_B$ because $\Pi$ is rigid.
\end{proof}

% The next result is a quantitative formulation of the negative correlation property of the determinantal point processes.
% 
% We begin with a definition:
% \begin{definition}
% \label{negcordef}
%  We call a point process on $\R^d$ negatively associated if, for any two disjoint Borel sets $A,B$ and integers $r,s\ge 0$ the number of points $N(A),N(B)$ satisfy \[ \P \big(   \l( N(A) \ge r \r) \cap \l( N(B) \ge s \r)  \big)  \le \P \big( N(A) \ge r \big) P \big( N(B) \ge s \big) \] 
% This is equivalent to the complementary condition 
% \[  \P \big(   \l( N(A) \le r \r) \cap \l( N(B) \le s \r)  \big)  \le \P \big( N(A) \le r \big)\P \big( N(B) \le s \big)  \]
% \end{definition}
% 
% In the theory of determinantal point processes on $\R^d$, by a standard kernel we mean a Hermitian kernel $\k(x,y)$ which is continuous as a function from $\R^d \times \R^d \to \C$ and is a non-negative trace class contraction when viewed as an integral operator from $L^2(\mu)\to L^2(\mu)$ where $\mu$ is the background measure. For further details, we refer the interested reader to \cite{HKPV}. 
% 
% \begin{theorem}
% \label{cna} 
% Any determinantal point process with a continuous kernel on $\R^d$ and the background measure $\mu$ absolutely continuous wrt Lebesgue measure is negatively correlated as in definition \ref{negcordef}. 
% \end{theorem}

We are now ready to establish the main Theorem \ref{rigcomp}.

% For Theorem \ref{rigcomp}, let $\pi$ be a translation invariant determinantal point process with determinantal kernel $K(\cdot,\cdot)$ on $\R^d$ with a background measure $\mu$ that is mutually absolutely continuous with respect to the Lebesgue measure on $\R^d$, such that $K(\cdot,\cdot)$ is also the projection kernel for the subspace $\H$ that is cannonically associated with $\pi$. Let the map $x\to K(\cdot,x)$ be continuous as a map from $\R^d \to \H$.
% \begin{theorem}
% \label{rigcomp}
% Let $\pi$ be rigid, in the sense that for any ball $B$, the point configuration outside $B$ a.s. determines the number of points $N_B$ of $\pi$ inside $B$. 
% % Furthermore, let the random variable $N_B$, for every ball $B \subset \R^d$, be regular at infinity. \newline 
% Then $\{K(\cdot,x): x \in \pi \}$ is a.s. complete in $\H$.
% \end{theorem}

\begin{proof}[Proof of Theorem \ref{rigcomp}]
Let $\H_0$ be the random closed subspace of $\H$ spanned by the functions $\{K(\cdot,x):x \in \Pi\}$ inside $L^2(\mu)$. We wish to show that a.s. we have $\H_0=\H$.
It suffices to prove that almost surely, we have $K(\cdot,\ze) \in \H_0$ for $\mu$-a.e. $\ze$. 
% This would imply that a.s. $K(\cdot,q) \in \H_0$ simultaneously for all $q \in$ a countable dense subset of the support of $\mu$. Then the continuity of the map $x \mapsto K(\cdot,x)$  on the support of $\mu$ would imply that $K(\cdot,x) \in \H_0$ simultaneously for all $x $ in the support of $\mu$ . This would imply that $\H_0=\H$ a.s. 
To see this, let $f \in \H$ be such that $f$ is orthogonal to $\H_0$. Since  $f \in \H$, therefore $f(x)=\int f(y)K(y,x) d\mu(y)$ for each $x$. But $K(\cdot,x)$ is in $\H_0$ for $\mu$-a.e. $x$, and $f$ is orthogonal to $\H_0$, so $f(x)=0$ for a.e. $x$ with respect to  $\mu$. In other words, $f \equiv 0$ in $L^2(\mu)$, therefore $\H_0=\H$.  

For any realization of $\Pi$, the set of functions $\{K(\cdot,x):x \in \Pi\}$ is trivially in $\H_0$. We need to show that a.s., we have $K(\cdot,x) \in \H_0$ for $\mu$-a.e. point $x \in E \setminus \Pi$. To this end, we will construct, for each $\eps>0$ and $x_0 \in E$ a pair of events $N(\eps,x_0)$ and $J(\eps,x_0)$ such that:
\begin{itemize} \item On the event $N(\eps,x_0)\setminus J(\eps,x_0)$, we have $K(\cdot,x) \in \H_0$ for $\mu$-a.e. $x$ in $B(x_0;\eps)$ \\ \item $\P(J(\eps,x_0))=0$ \\ \item For any $x \notin \Pi$, the event $N(\eps,x)$ occurs for each value of $\eps < d(x,\Pi)$. \end{itemize}

Here $d(x,\Pi)$ is the distance between $x$ and the set $\Pi$ and $B(x_0;\eps)$ is the open ball with centre $\ze$ and radius $\eps$.

We claim that such events $N(\eps,x_0)$ and $J(\eps,x_0)$ are enough to ensure that a.s., we have $K(\cdot,x) \in \H_0$ for $\mu$-a.e. point $x \in E \setminus \Pi$. We will prove this claim in the following paragraphs.

Consider a countable dense set of points $\q$ in $E$. For each point $q \in \q \setminus \Pi$, consider $B(q;r_q)$ where $r_q$ is any rational number between $d(q,\Pi)$ and $d(q,\Pi)/2$. 
We first show that these balls cover $E \setminus \Pi$. Any point $x \in E \setminus \Pi$ is a limit of a sequence of points belonging $\q$. Since $\Pi$ is a locally finite set of points and $x \in E \setminus \Pi$, this sequence will eventually be contained in $\q \setminus \Pi$. If we denote this sequence by $\{q_k\}_{k \ge 0}$, we  have $q_k \to x$ as $k \to \infty$ and $\varliminf_{k \to \infty} r_{q_k} \ge d(x,\Pi)/2 >0$. Thus, $x \in \bigcup_{q \in \q} B(q;r_q)$, as desired. 

Since \[\P\l(\cup_{q \in \q} \cup_{q \in \mathbb{Q}} J(r,q) \r) =0,\] we can assume that we are working on the complement of the event $\cup_{q \in \q} \cup_{q \in \mathbb{Q}} J(r,q) $. Then for each $q \in \q \setminus \Pi$, the event $N(r_q,q)$ occurs. Therefore, a.s. the event $\cup_{q \in \q \setminus \Pi} N(r_q,q)$ occurs. Each event $N(r_q,q)$ implies that for $\mu$-a.e. point $x$ in $B(q;r_q)$ the function $K(\cdot,x)$ is in $\H_0$.  But we have already seen that the sets $B(q;r_q)$ cover $E \setminus \Pi$. Therefore, a.s. it is true that for  $\mu$-a.e. point $x$ in $E \setminus \Pi$, the function $K(\cdot,x)$ is in $\H_0$. This completes the proof of our claim.

We now proceed to construct the events $N(\eps,x_0)$ and $J(\eps,x_0)$ for each $x_0 \in E$ and $\eps>0$. The event $N(\eps,x_0)$ is simply the event that $B(x_0;\eps)\cap \Pi = \phi$. It is obvious from the definition that whenever $x_0 \notin \Pi$, for any $\eps<d(x_0,\Pi)$ the event $N(\eps,x_0)$ occurs. In what follows, we will show that for a fixed $x_0$ and $\eps$, a.s. on the event $N(\eps,x_0)$ the function $K(\cdot,x)$ is in $\H_0$ for $\mu$-a.e. point $x$ in $B(x_0;\eps)$. The null event on which the statement does not hold will be our event $J(\eps,x_0)$. 

For brevity, we will henceforth denote by $B_r$ the open ball of radius $r>0$ centred at $\ze$.

For the points $\{x_i\}_{i=1}^n \in E$, let ${\mathcal{D}}(x_1,\cdots,x_n)$ denote  $\mbox{ det } \bigg[ \l( K(x_i,x_j)\r)_{i,j=1}^n \bigg]$. Then, if we consider the function $ K(\cdot,x) $ as a vector in the Hilbert space $\H$, the squared norm of the projection of $K(\cdot,x)$ onto the  orthogonal complement of $\mbox{Span }\{K(\cdot,x_i), 1\le i \le n\}$ is given by the ratio $\frac{\mathcal{D}(x,x_1,\cdots,x_n)}{\mathcal{D}(x_1,\cdots,x_n)}$. This follows easily from the interpretation of the determinant of a Gram matrix as the squared volume of a parallelopiped. But $\frac{\mathcal{D}(x,x_1,\cdots,x_n)}{\mathcal{D}(x_1,\cdots,x_n)}$ is also equal to the conditional intensity $p(x|x_1,\cdots,x_n)$ (with respect to the background measure $\mu$) of $\Pi$ at $x$ given that $\{x_1,\cdots,x_n\}\subset \Pi$ , see e.g. Corollary 6.6 in \cite{ShTa}. Let us elaborate a little more on this conditioning. Let $R > \eps$. We fix a point configuration $\u_R=\{x_1,\cdots,x_n\}$ in $\ol{B_R} \setminus B_{\eps}$ and consider the point process $\Pi''$  obtained by conditioning $\Pi$ to contain $\u_R$ (with the understanding that $\Pi''$ contains all the points in $\u_R$).  Now, set $\Pi' = \Pi '' \setminus \u_R$. Let $\P_{\u_R}$ be the law of the point process $\Pi'$. Such conditioning is well known in the literature as the Palm measure of $\Pi$ at $\u_R$. For more details on the Palm measure, we refer to \cite{Ka} Chapter 10 or \cite{ShTa} Section 6. 

Let $\o_R$ be the random point configuration obtained by restricting  $\Pi$ to $\ol{B_R} \setminus B_{\eps}$. For any  fixed point configuration $\u_R$ in $\ol{B_R} \setminus B_{\eps}$ we have 
\begin{equation}\label{step1} \E_{\U} [\mbox{ Number of points in }B_{\eps} ] = \int_{B_{\eps}} p(x|\u_R)d\mu(x)  \end{equation}
We now proceed with the left hand side in (\ref{step1}) as 
\begin{align*}
&\E_{\U} [\mbox{ Number of points in }B_{\eps} ] \\ &= \sum_{k=1}^{\infty} \P_{\U} [\mbox{ There are } \ge k \mbox{ points in }B_{\eps}] \\
&\le \sum_{k=1}^{\infty} \P [\mbox{ There are } \ge k \mbox{ points in }B_{\eps} | \o_R=\u_R ] \mbox{ (see Proposition } \ref{techprop1})\\
&=\E [\mbox{ Number of points in }B_{\eps} | \o_R=\u_R  ]
\end{align*}

% The inequality \[\P_{\U} [\mbox{ There are } \ge k \mbox{ points in }B_{\eps}  ] \le \P [\mbox{ There are } \ge k \mbox{ points in }B_{\eps} | \o_R=\u_R ]\]
is a consequence of the negative association property of determinantal point processes, and is proved in Proposition \ref{techprop1} using Theorem \ref{cna}.

Denote by $N_{B_{\eps}}$ the number of points of $\Pi$ in $B_{\eps}$.
By Proposition \ref{regularity}, we have that given any $\delta>0$, we can find $R_{\delta}$ such that except on an event $\Omega_1^{\delta}$ of probability $<\delta$, we have \[\l|\E [N_{B_{\eps}} | \o_{R_{\delta}} ] - N_{B_{\eps}} \r|<\delta.\]

% But, except on an event $\Omega_2^{\eps}$ (which is measurable with respect to the configuration  $\ou$ of points of $\Pi$ outside $B_{\eps}$ ), 
But, on the event $N(\eps,\ze)$, we have $N_{B_{\eps}}=0$. 
% Further, $\P(\Omega_2^{\eps} \cap A_{\ze}^{\complement})=o(1)$ as $\eps \to 0$. To see this, note that due to rigidity the event $\Omega_2^{\eps}:=\{N_{B_{\eps}}>0\}$ is measurable with respect to $\ou$, and satisfies $\lim_{\eps \downarrow 0} \Omega_2^{\eps} = A_{\ze}$. Therefore, \[ \lim_{\eps \downarrow 0} \P(\Omega_2^{\eps} \cap  A_{\ze}^{\complement})=\P[A_{\ze} \cap  A_{\ze}^{\complement}]=0.\] 
% But $x \mapsto K(\cdot,x)$ is continuous, so $x \mapsto ||K(\cdot,x)||_2^2 = K(x,x)$ is continuous. Since $\mu$ does not have atoms, $\int_{B_{\eps}} K(x,x) d\mu(x) \to 0$ as $\eps \to 0$. Hence $\P(\Omega_2^{\eps}) \to 0$ as $\eps \to 0$.
% $\E[N_{B_{\eps}}] \to 0$ as $\eps \to 0$ because the determinantal process $\Pi$ does not have any atoms (that is, the probability of having a point of $\Pi$ at any specified location is 0). 

Hence, on the event $N(\eps,\ze) \setminus \Omega_1^{\delta}$, we have 
\[\int_{B_{\eps}} p(x|\o_{R_{\delta}})d\mu(x) \le \delta \]
Letting $\delta \downarrow 0$ along a summable sequence and $R_{\delta} \uparrow \infty$, we deduce that a.s. on $N(\eps,\ze)$ we have,  \[\lim_{\delta \to 0}\int_{B_{\eps}} p(x|\o_{R_{\delta}})d\mu(x)=0.\] This is true because, by the Borel Cantelli lemma, the event $J(\eps,x_0)$ that $\Omega_1^{\delta}$ occurs infinitely often (along this summable sequence of $\delta$-s) has probability zero. 

By Fatou's lemma, this implies that on the event $N(\eps,\ze)$, a.s. we have \[\int_{B_{\eps}} \varliminf_{\delta \to 0} p(x|\o_{R_{\delta}})d\mu(x)=0.\] This implies that on the event $N(\eps,\ze)$,  a.s. we have $\varliminf_{\delta \to 0}  p(x|\o_{R_{\delta}})=0$ for almost every $x \in B_{\eps}$ (with respect to the measure $\mu$). By our previous discussion, at the beginning of the proof, regarding the connection between the squared norms of projections and conditional intensities, this means that on the event $N(\eps,\ze)$, a.s. we have  $K(\cdot,x) \in \H_0$ for a.e.  $x \in B_{\eps}$ (with respect to $\mu$), as desired.
% But $\ze \in B_{\eps}$. By the continuity of the map $x \mapsto K(\cdot,x)$ on the support of $\mu$, we have $K(\cdot,\ze) \in \H_0$ a.s. on the event $(\Omega_2^{\eps})^{\complement} \cap A_{\ze}^{\complement}$. Letting $\eps \downarrow 0$ (which implies $\Omega_2^{\eps} \downarrow A_{\ze}$), we have $K(\cdot,\ze) \in \H_0$ a.s. on $A_{\ze}^{\complement}$, as desired.
\end{proof}

We now  establish the precise negative association inequality necessary for the above theorem :

\begin{proposition}
 \label{techprop1}
 \[\P_{\U} [\mbox{ There are } \ge k \mbox{ points in }B_{\eps} ] \le \P [\mbox{ There are } \ge k \mbox{ points in }B_{\eps} | \o_R=\u_R ]\] for almost every  configuration $\u_R$, the probability measure on $\u_R$ being that induced by the random variable $\o_R$.
\end{proposition}

\begin{proof}
  Recall from the proof of Theorem \ref{rigcomp} that $\P_{\U}$ is the law of the point process $\Pi'$ .
  It is known that  $\Pi'$ is again a determinantal point process on $E$ with background measure $\mu$, see \cite{ShTa} Corollary 6.6.  
  Applying Theorem \ref{cna} to $\Pi'$, we get that \[\P_{\u_R} \l( \mbox{ There are } \ge k \mbox{ points of } \Pi' \mbox{ in } B_{\eps} \cap  \mbox{ There is }  \mbox{ no point of } \Pi' \mbox{ in } \ol{B_R} \setminus B_{\eps} \r)\]\[ \ge \P_{\u_R} \l( \mbox{ There are } \ge k \mbox{ points of } \Pi' \mbox{ in } B_{\eps} \r)  \P_{\u_R} \l( \mbox{ There is }  \mbox{ no point of } \Pi' \mbox{ in } \ol{B_R} \setminus B_{\eps} \r). \]

Since $\u_R$ is  a  realization of the full point configuration $\o_R$ of $\Pi$ in $\ol{B_R} \setminus B_{\eps}$, therefore $\P_{\u_R} \l( \mbox{ There is }  \mbox{ no point of } \Pi' \mbox{ in } \ol{B_R} \setminus B_{\eps} \r) >0$ (for almost every such configuration $\u_R$; the probability measure on $\u_R$ being that induced by the random variable $\o_R$).

Hence the last inequality can be rephrased as  \[  \P_{\u_R} \l( \mbox{ There are } \ge k \mbox{ points of } \Pi' \mbox{ in } B_{\eps} \r) \] \[\le \P_{\u_R} \l( \mbox{ There are } \ge k \mbox{ points of } \Pi' \mbox{ in } B_{\eps} | \mbox{ There is }  \mbox{ no point of } \Pi' \mbox{ in } \ol{B_R} \setminus B_{\eps} \r)\]\[= \P\l( \mbox{ There are } \ge k \mbox{ points of } \Pi' \mbox{ in } B_{\eps} | \o_R = \u_R \r).\]
% Under the conditioning $ \u_R \subset \Pi $, the event that there is no point of $\Pi'$ in  $\ol{B_R} \setminus B_{\eps}$ corresponds to $\o_R=\u_R$. 
The last equality is a relation between the Palm measure and the conditional probability measure of $\Pi$ under the conditioning $\o_R = \u_R$. While this relation is intuitively straightforward, it is explained and proved in Proposition \ref{techprop2}. This completes the proof of our desired inequality.

\end{proof}

 \begin{proposition}
  \label{techprop2}
  $\P_{\U}( M | \mbox{ There is }  \mbox{ no point of } \Pi' \mbox{ in } \ol{B_R} \setminus B_{\eps}  ) = \P( M | \o_R = \U)$, where $M$ is an event which depends only on the points in $B_{\eps}$.
 \end{proposition}

 \begin{proof}
  We will prove this, for ease of demonstration, in the case $|\U|=1$; the general case follows on similar lines. In this setting, we will prove a statement connecting the Palm and conditional measures of a general point process, of which our  situation is a special case. Here $|\U|$ is the number of points in $\U$. 
  
  Consider a point process $\xi$ on the second countable locally compact Hausdorff space $E$  and a Borel set $V \subset E$. Let $\rho_1$ denote the first intensity measure of $\xi$. Let $N$ denote the event that there are no points in $V$, and $M$ be any event that depends only on the points in $V^{\complement}$.  We will denote the point process $\xi$ conditioned to contain $s$ to be $\xi_s$. By definition, $\xi_s$ contains the point $s$. We will denote by $\del_s$ the delta measure at $s$. Let $\P$ denote the law of $\xi$ and $\P_s$  denote the law of $\xi_s - \del_s$. Let $\xi^V$ denote the restriction of $\xi$ to the set $V$.  We want to show that, for $\rho_1$-a.e. point $x$ in $V$ (such that $\P_x(N)>0$), we have $\P_x(M|N)=\P(M|\xi^V=\{x\})$.
  
%   Let $N$ denote the event that there are no points of a point process in $V$, and $M$ be any event that depends only on the points in $B_{\eps}$. 
  Let $h$ be any non-negative real valued compactly supported measurable function whose support is contained in $V$. 
%   Let $\rho_1$ denote the first intensity measure of the point process $\xi$.
  It suffices to show that \[\int h(s) \P_s(M|N) \P_s(N) \rho_1(ds) = \int h(s) \P(M|\xi^V=\{s\}) \P_s(N) \rho_1(ds)   \] for all such $h$.
  In what follows, we will make repeated use of the following relation, which holds for any measurable function $f: E \times \mathcal{N} \to \mathbb{R}_{+}$ (where $\mathcal{N}$ denotes the space of counting measures corresponding to locally finite point configurations on $E$ and $\R_+$ denotes the non-negative real numbers). In this setting, we have \begin{equation}
                                                               \label{relat}
                                                               \l[\E f(s,\xi_s -\del_s)\r] \rho_1 (ds) = \E \l[ f(s,\xi - \del_s) \xi(ds) \r] .
                                                             \end{equation}
  This is a restatement of Lemma 10.2 in \cite{Ka}. This also follows from (6.1) in \cite{ShTa} (for the situation $|\U|>1$ look at (6.5) in \cite{ShTa}). Here $\xi(ds)$ refers to the counting measure induced by the particular realization of the point process $\xi$.
  
  We have,
  \[  \int h(s) \P_s(M|N) \P_s(N) \rho_1(ds) = \int h(s) \P_s(M \cap N) \rho_1(ds) \]\[  = \int  \E [ h(s) 1_{M \cap N}(\xi_s -\del_s)] \rho_1(ds)   = \E \l( \int h(s) 1_{M\cap N}(\xi -\del_s) \xi(ds) \r). \] In the last equality we have used relation (\ref{relat}). Observe that \[1_{M\cap N}(\xi -\del_s)= 1_M(\xi)1(\xi^V = \{s\}),\] because the event $M$ depends only on the points in $V^{\complement}$ and $s \in V$. Since $h$ is supported on $V$, we have \[  \int h(s) 1_{M\cap N}(\xi -\del_s) \xi(ds) = h(\xi^V) 1_M(\xi) 1(|\xi^V|=1) .\] Therefore,  \[ \int h(s) \P_s(M|N) \P_s(N) \rho_1(ds) =  \E \l( \int h(s) 1_{M\cap N}(\xi -\del_s) \xi(ds) \r)    \]
  \begin{equation}
  \label{techrel}                                                                                                                                                                                                                                                                                                                                                                                                                                                                                                                                                                                                                                                                  
   = \E \l[ h(\xi^V) 1_M(\xi) 1(|\xi^V|=1) \r] = \int_{V}     h(s) \P(\xi \in M | \xi^V=\{s\}) d\P_{\xi^V}(\{s\}) 
   \end{equation}                                                                                                                                                                                                                                                                                                                                                                                                                                                                                                                                                                                                                                                      where $\P_{\xi^V}$ is the marginal distribution of $\xi^V$.   All that remains to show, therefore, is that on the event $\{|\xi^V|=1\}$, we have the following equality of measures (on the set $V$): \[   d\P_{\xi^V}(\{s\})=\P_s(N)\rho_1(ds).\]  However, this follows from  relation (\ref{techrel}), as discussed below.
%   By setting $h(s)=\theta(s)$ in (\ref{techrel}), where $\theta$ is a non-negative function compactly supported on $V$ and setting $M=\Omega$ (the universal event), we get \[  \int_{V} \theta(s) \P_s(N) \rho_1(ds) = \E \l[\theta(\xi^V) 1({|\xi^V|=1})\r] = \int_{V} \theta(s) d\P_{\xi^V}(\{s\}).  \] This holds for all such $\theta$, giving us the desired equality of measures. This completes the proof.
 By setting  $M=\Omega$ (the universal event) in (\ref{techrel}), we get \[  \int_{V} h(s) \P_s(N) \rho_1(ds) = \E \l[h(\xi^V) 1({|\xi^V|=1})\r] = \int_{V} h(s) d\P_{\xi^V}(\{s\}).  \] This holds for all such $h$, giving us the desired equality of measures. This completes the proof.
\end{proof}

\section{Rigidity and Tolerance for certain determinantal point processes} 
\label{rig}
% Let $f$ be a function $\td \to [0,1]$. Then multiplication by $f$ is a non-negative contraction operator from $L^2(\td)$ to itself. Under Fourier conjugation, this gives rise to a non-negative contraction opertaor   $Q: \ell^2(\Z^d) \to \ell^2(\Z^d)$. This, in turn, gives rise to a determinantal point process $\pf$ on $\Z^d$ in a cannonical way, for details see \cite{Ly}. For a point configuration $\o$ drawn from the distribution of $\pf$, we denote by $\o(\underline{k})$ the indicator function of having a point at $\underline{k}\in \Z^d$ in the configuration $\o$. We denote by $\ou$ the configuration  of points on $\Z^d \setminus \ze$ obtained by restricting $\o$ to $\Z^d \setminus \ze$, where $\ze$ denotes the origin.  For a point process $\pi$ on a space $\Xi$, we denote by $[\pi]$ the (random) counting measure obtained from a realisation of $\pi$. The $k$-point intensity functions of a point process, when it exists, will be denoted by $\rho_k, k \ge 1$. For the processes $\pf$, all intensity functions 
% exist. Moreover, translation invariance of $\pf$ implies that $\rho_k(x_1+x,\cdots,x_k+x)=\rho_k(x_1,\cdots,x_k)$ for all $x,x_1,\cdots,x_k \in \Z^d$, in particular, $\rho_1$ is a constant in $[0,1]$.
% 
% In the paper \cite{LySt} it was conjectured that all determinantal processes obtained in this way are insertion and deletion tolerant, meaning that $\P[\o(\ze)=1 | \ou ]>0$  and $\P[\o(\ze)=0 | \ou ]>0$. We answer this question in the negative, showing that for $f$ which is the indicator function of an interval, this is not true. 

In this section we discuss the insertion and deletion tolerance question from \cite{LySt}, principally the proof of Theorem \ref{indic}.

We will use the following general observation for determinantal point processes given by a projection kernel:

\begin{proposition}
\label{varlst}
Consider a determinantal point process $\Pi$ on a  locally compact space $\Xi$ with determinantal kernel $K(\cdot,\cdot)$ and background measure $\mu$, such that $K$ is  a projection as an integral operator on $L^2(\mu)$. Let $\psi$ be a compactly supported function on $\Xi$.  Then \begin{equation} \label{vlst} \var \l[ \int \psi \, d[\Pi] \r] = \frac{1}{2} \iint |\psi(x)-\psi(y)|^2 |K(x,y)|^2 \, d\mu(x) d\mu(y)   \end{equation}
\end{proposition}

\begin{proof}
Denote by $\rho_1$ and $\rho_2$ respectively the one and two-point correlation functions of $\Pi$. Then we can write
\begin{equation}
\label{vlst1} 
\var \l[ \int \psi \, d[\Pi] \r] = \int \psi(x)\ol{\psi(y)} \rho_2(x,y) d\mu(x)d\mu(y) + \int \l|\psi(x)\r|^2 \rho_1(x) d\mu(x) - \l|\int \psi(x) \rho_1(x) d\mu(x)\r|^2.
\end{equation}
But $\rho_1(x)=K(x,x)$ and $\rho_2(x,y)=K(x,x)K(y,y) - |K(x,y)|^2$. Using this, the expression for the variance in (\ref{vlst1}) reduces to 
\begin{equation} \label{vlst2}   \iint |\psi(x)|^2 K(x,x) d\mu(x)  - \iint \psi(x)\ol{\psi(y)} |K(x,y)|^2 d\mu(x)d\mu(y). \end{equation}
But since $K$ is the integral kernel corresponding to a projection operator and $K(y,x)=\ol{K(x,y)}$, we have $K(x,x)=\iint |K(x,y)|^2 d\mu(y)$. 
Using this, the expression for the variance in (\ref{vlst2}) reduces to 
\[ \iint \l(  |\psi(x)|^2 - \psi(x)\ol{\psi(y)}  \r) |K(x,y)|^2 d\mu(x) d\mu(y) =\frac{1}{2} \iint |\psi(x)-\psi(y)|^2 |K(x,y)|^2 \, d\mu(x) d\mu(y),  \] as desired. In the last step, we have used  symmetry in $x$ and $y$.
\end{proof}

\begin{proof}[Proof of Theorem \ref{indic}]
We will approach this question by estimating the variance of linear statistics of $\pf$. A similar approach has been used in \cite{GNPS} to obtain rigidity behaviour for the Ginibre ensemble and the zero process of the standard planar Gaussian analytic function. 

Let $\ph$ be a $C_c^{\infty}$ function on $\R$ which is $\equiv 1$ in a neighbourhood of the origin. Viewed as a  function on $\Z$,   $\ph$ is compactly supported and $=1$ at the origin. Let $\ph_L$ be defined by $\ph_L(x)=\ph(x/L)$. Since $f$ is the indicator function of an interval $I \subset \t$, therefore the determinantal kernel $K$ of $\pf$ gives rise to a projection as an integral operator on $\ell^2(\Z)$.  Applying Proposition \ref{varlst} with $\Pi=\pf$, $\Xi=\Z$, $\mu=\text{the counting measure on }\Z$ and $\psi=\ph_L$ we get 
\begin{equation}
\label{var}                                                                                                                                                                                                                                                                                                                                                                                                                                                                                                                                                    
\var \l[ \int \ph_L \, d[\pf] \r] = \frac{1}{2} \sum_{i,j \in \Z} |\ph(\frac{i}{L})-\ph(\frac{j}{L})|^2 |\hat{f}(i-j)|^2 .                                                                                                                                                                                                                                                                                                                                                                                                                                                                                                                                                   \end{equation}
Observe that translating the interval $I \subset \t$ leaves the measure $\pf$ invariant, so, without loss of generality, we take $I$ to be corresponding to the interval $[-a,a]$ where $\t$ is parametrized as $(-\pi,\pi]$ and $0 < a < \pi$. Then $\hat{f}(k)=c(a) {\sin{ak}}/{k}$ where $c(a)$ is a constant. This implies that, for some constant $c>0$, we have 
\begin{equation}
\label{est1} 
\var \l[ \int \ph_L \, d[\pf] \r] = c \sum_{i,j \in \Z} |\ph(\frac{i}{L})-\ph(\frac{j}{L})|^2 \l( \sin^2a(i-j) \r) |i-j|^{-2} .  
\end{equation}
This, in turn, implies (using $|\sin \theta| \le 1$) that
\begin{equation}
\label{est2} 
\var \l[ \int \ph_L \, d[\pf] \r] \le c \sum_{i,j \in \Z} |\ph(\frac{i}{L})-\ph(\frac{j}{L})|^2  |(i-j)/L|^{-2} L^{-2}.  
\end{equation}
Hence we have 
\begin{equation}
\label{est3} 
\varlimsup_{L \to \infty} \var \l[ \int \ph_L \, d[\pf] \r] \le c \int \int \l(\frac{\ph(x)-\ph(y)}{x-y}\r)^2 \, d\leb(x)d\leb(y) 
\end{equation}
where $\leb$ denotes the Lebesgue measure on $\R$.

For any $\C_c^1$ functions $\psi_1,\psi_2$ on $\R$,  we define the form \begin{equation} \label{form} \L(\psi_1,\psi_2)=\int \int \frac{\l(\psi_1(x)-\psi_1(y)\r) \l( \psi_2(x) - \psi_2(y) \r)}{(x-y)^2} \, d\leb(x)d\leb(y).      \end{equation}
It is known that $\L(\psi,\psi)$ is related  to the $H^{1/2}$ norm of $\psi$. 

A simple calculation shows that for any $\la>0$, we have $\L((\psi_1)_{\la},(\psi_2)_{\la})=\L(\psi_1,\psi_2)$. In particular, this implies that $\L(\psi_1,(\psi_2)_{\la})=\L((\psi_1)_{1/\la},\psi_2)$. Further, we will see in Proposition \ref{decay} that $\L(\psi,\psi_{{\la}^{-1}}) \to 0$ as $\la \to 0$. 

For an integer $n>0$, let $0<\la=\la(n)<1$ be such that for $\ph$ as above, $|\L(\ph,\ph_{\la^{-i}})|\le 1/2^i$ for  $1 \le i \le n$. Such a choice can be made because of the observations in the previous paragraph. Define $\Phi^n=\l( \sum_{i=1}^n  \ph_{\la^{-i}} \r)/n$.  Note that $\Phi^n \equiv 1$ in a neighbourhood of $\mathbf{0}$ in $\R$, and the same is true for all scalings $\Phi^n_L$ of $\Phi^n$ whenever $L \ge 1$. 

Let $L=L(n)>1$ be such that 
\[ \var \l[ \int \Phi^n_L \, d[\pf] \r] \le c \L(\Phi^n,\Phi^n) + \frac{1}{n}. \]
But $\L(\Phi^n,\Phi^n)=\frac{1}{n^2}\l( \sum_{i,j=1}^n \L(\ph_{\la^{-i}}, \ph_{\la^{-j}} )  \r)$. Observe that \[ \L(\ph_{\la^{-i}}, \ph_{\la^{-j}} )= \L(\ph,\ph_{\la^{-|i-j|}}) \le 2^{-|i-j|}.\] This implies that \[\sum_{i,j=1}^n \L(\ph_{\la^{-i}}, \ph_{\la^{-j}}) \le C(\ph)n.\] Hence  $ \var \l[ \int \Phi^n_L \, d[\pf] \r] \le C(\ph)/n$.

By the Borel-Cantelli lemma, we have, as $n \to \infty$, \begin{equation} \label{conc} \l|  \int \Phi^{2^n}_L \, d[\pf] - \E [ \int \Phi^{2^n}_L \, d[\pf]] \r| \to 0. \end{equation}

But $ \int \Phi^{2^n}_L \, d[\pf] = \o(\mathbf{0})+ \int_{\Z \setminus \mathbf{0}} \Phi^{2^n}_L \, d[\pf]$, and the second term can be evaluated if we know $\ou$. $\E [ \int \Phi^{2^n}_L \, d[\pf]]$ can also be computed explicitly in terms of the first intensity measure of $\pf$. This implies that from (\ref{conc}), we can deduce the value of $\o(\mathbf{0})$ by letting $n \to \infty$.

Thus,  $\ou$ a.s. determines the value of $\o(\mathbf{0})$. Since both the events $\o(\mathbf{0})=0$ and $\o(\mathbf{0})=1$ occur with positive probability, therefore  the events  $\{\P[\o(\mathbf{0})=1 | \ou ]=0\}$ and  $\{\P[\o(\mathbf{0})=0 | \ou ]=0\}$ both have positive probability (in $\ou$).  
\end{proof}

\begin{remark}
For $f$ which is the indicator function of a finite, disjoint union of intervals $\subset \t$, we have $|\hat{f}(k)|\le c/|k|$, hence the same argument and the same conclusion as Theorem \ref{indic} holds for such $f$.
\end{remark}

\begin{remark}
A similar argument  shows that, in fact, for any finite set $S \subset \Z$, the point configuration of $\pf$ restricted to  $S^{\complement}$ a.s. determines the number of points of $\pf$ in $S$, when $f$ is the indicator function of an interval.
\end{remark}

A similar class of determinantal point processes in the continuum can be obtained by considering $L^2$ functions $f:\R^d \to [0,1]$. The multiplication operator $M_{f}$ defined by such a function $f$ is clearly a contraction on $L^2(\R^d)$. By considering the Fourier conjugate of such an operator, we get another contraction on $L^2(\R^d)$, which gives us a translation invariant determinantal point process $\pf$ in $\R^d$. One of the most important examples of such a point process is the sine kernel process on $\R$, which is defined by the determinantal kernel $\frac{\text{sin}\pi(x-y)}{\pi(x-y)}$ with the Lebesgue measure on $\R$ as the background measure. Here the relevant function $f$ is the indicator function of the interval $[-\pi,\pi]$. For details, see \cite{AGZ}. More generally we can consider the indicator function of any measurable subset of $\R$, which will give us a projection operator on $L^2(\R)$, and hence a determinantal point process corresponding to a projection kernel. In this setting, we 
have a continuum analogue of Theorem \ref{indic}, which says that whenever $f$ is the indicator of a finite union of compact intervals in $\R$, we have that $\pf$ is a rigid process.

\begin{theorem}
\label{contrig}
Let $f:\R \to [0,1]$ be an indicator function of a finite union of compact intervals. Then the determinantal point process $\pf$ in $\R$ is ``rigid'' in the following sense. Let $U \subset \R$ be a bounded interval, and let $\o$ be the point configuration sampled from the distribution $\pf$. Define the restricted point configurations $\oin=\o_{|U}$ and $\ou=\o_{|U^{\complement}}$. Let $|\oin|$ be the number of points of $\o$ in $U$. Then there exists a measurable function
\[ N:\mbox{ Point configurations in } U^{\complement} \to \N \cup \{0\}\] 
such that a.s. we have $|\oin|=N(\ou)$. This holds true for all bounded intervals $U$.    \newline
In particular, the continuum sine kernel process is ``rigid'' in the above sense. 
\end{theorem}

\begin{proof}
By translation invariance, it suffices to take $U$ to be centred at the origin. Let $\ph$ be a $C_c^{\infty}$ function which is $\equiv 1$ in a neighbourhood of $U$. 
Then we have 
\begin{equation}
 \label{variance1}
\var \l[ \int \ph \, d[\pf] \r] = \frac{1}{2} \int \int |\ph(x) -\ph(y)|^2 |\hf(x-y)|^2 d\leb(x)d\leb(y).
\end{equation}
But for any compact interval $[a,b]$ we have $|\hat{1}_{[a,b]}(\xi)| \le c|\xi|^{-1}$, hence we have
\begin{equation}
 \label{variance2}
\var \l[ \int \ph \, d[\pf] \r] \le C\int \int \l(\frac{\ph(x) -\ph(y)}{x-y}\r)^2  d\leb(x)d\leb(y).
\end{equation}
Recall the form $\L(\cdot,\cdot)$ as in (\ref{form}). Proposition \ref{decay} implies that $\L(\ph,\ph_{\la^{-1}}) \to 0$ as $\la \to 0$.
For an integer $n>0$, let $0<\la<1$ be such that for $\ph$ as above, $|\L(\ph,\ph_{\la^{-i}})|\le 1/2^i$ for  $1 \le i \le n$. Define $\Phi^n=\l( \sum_{i=1}^n  \ph_{\la^{-i}} \r)/n$. We have $\Phi^n \equiv 1$ in a neighbourhood of $U$ in $\R$. 
Due to our choice of $\la$, we have $\var \l( \int \Phi^n d[\pf] \r)=\L(\Phi^n,\Phi^n)=O(1/n)$. From here, we proceed on similar lines to the proof of Theorem \ref{indic} and deduce the existence of $N$ as prescribed in the statement of Theorem \ref{contrig}. 
\end{proof}

We now prove Proposition \ref{decay}, which will complete the proof of Theorem \ref{indic}.

\begin{proposition}
\label{decay}
For a $C_c^1$ function $\ph$, we have $\L(\ph,\ph_{\la^{-1}})\to 0$ as $\la \to 0$.
\end{proposition}

\begin{proof}
We begin with the expression \begin{equation}\label{dec1} \L(\ph,\ph_{\la^{-1}})=\int \int \frac{\l(\ph(x)-\ph(y)\r) \l( \ph(\la x) - \ph(\la y) \r)}{(x-y)^2} \, d\leb(x)d\leb(y). \end{equation}
Fix $a>0$. Let $K$ be the support of $\ph$. We define the function 
\begin{equation}
\label{dec2} 
\gamma(x,y)= \left \{
                      \begin{array} {lll}
                    \|  \ph' \|_{\infty}^2 & \mbox{if } x \mbox{ or } y \in K \mbox{ and } |x-y|\le a \\
                    \frac{4\| \ph \|_{\infty}^2}{(x-y)^2} & \mbox{if } x \mbox{ or } y \in K \mbox{ and } |x-y|> a \\
                    0 & \mbox{otherwise }
                      \end{array}              
              \right.
\end{equation}
For $0<\la<1$, the integrand in (\ref{dec1}) is bounded in absolute value from above pointwise by $\gamma(x,y)$. To see this, note that the integrand in (\ref{dec1}) is non-zero only on the set $S=\{ (x,y): x \mbox{ or } y \in K \}$. On $S$, we bound the integrand from above as follows: for $(x,y)\in S$ such that $|x-y|\le a$ we use $|\ph(x)-\ph(y)|\le \|\ph'\|_{\infty}|x-y|$, for other $(x,y)\in S$ we use $|\ph(x)-\ph(y)|\le 2\|\ph\|_{\infty}$.

Since $K$ is a compact set, we have \begin{equation} \label{dec3} \int \int \gamma(x,y) \, d\leb(x) d\leb(y) <\infty, \end{equation} where we use the fact that $\int_{|t|>a}\frac{1}{t^2}dt = \frac{2}{a}$.  

(\ref{dec3}) enables us to use the dominated convergence theorem and let $\la \to 0$ in the integrand of (\ref{dec1}), whence $\L(\ph,\ph_{\la^{-1}}) \to 0$.
\end{proof}

Next we show that in any dimension $d$, whenever $f$ is not the indicator of a subset of $\t^d$, $\var \l[ \int \ph_L d[\pf] \r]$ blows up at least like $L^d$ as $L \to \infty$. So, the above approach cannot be used to decide questions on rigidity phenomena in such situations.

\begin{proposition}
\label{explo} 
Let $f:\t^d \to [0,1]$ not equal the  indicator function of some subset of $\t^d$ (up to Lebesgue-null sets). Then $\var \l[ \int \ph_L d[\pf] \r] = \Omega(L^d)$ as $L \to \infty$.
\end{proposition}
\begin{proof}
Let $\rho_2(\cdot,\cdot)$ be the two point intensity function of $\pf$, given by the formula \[\rho_2(i,j)= \mbox{det}\l( \begin{array} {cc} \hat{f}(0) & \hat{f}(i-j) \\ \hat{f}(j-i) & \hat{f}(0) \end{array} \r) \] Let $\la_d$ denote the normalized Lebesgue measure on $\td$.
Using the above formula, we can write the variance in question as: 
\begin{align*}
\var \l[ \int \ph_L d[\pf] \r] &= \E \l(\int \ph_L d[\pf] \r)^2 - \l( \E \l[ \int \ph_L \,d[\pf] \r] \r)^2 \\
&= \sum_i \ph_L(i)^2 \hf(0) + \sum_{i,j} \ph_L(i) \ph_L(j) \l( \hf(0)^2 - |\hf(i-j)|^2 \r) - \sum_{i,j} \ph_L(i)\ph_L(j)\hf(0)^2 \\
&= \sum_i \ph_L(i)^2 \l( \hf(0) - \sum_j |\hf(i-j)|^2 \r) + \sum_{i,j} \l( \ph_L(i)^2 - \ph_L(i)\ph_L(j)\r)|\hf(i-j)|^2 \\
&= \l(\hf(0) - \sum_k |\hf(k)|^2 \r) \l( \sum_k \ph_L(k)^2  \r) + \frac{1}{2} \l( \sum_{i,j} \l| \ph_L(i) - \ph_L(j) \r|^2 |\hf(i-j)|^2 \r) \\
&\ge \l(\hf(0) - \sum_k |\hf(k)|^2 \r) \l( \sum_k \ph_L(k)^2  \r) \\
&= \l( \int_{\td} f(x) d\la_d(x) - \int_{\td} f(x)^2 d\la_d(x) \r)\l( \sum_k \ph_L(k)^2  \r). 
\end{align*}
In the last step we have used Parseval's identity: $\sum_k |\hf(k)|^2 = \int f(x)^2 d\la_d(x)$.
Note that since $0\le f \le 1$, we have $\bigg( \int_{\td} f(x) d\la_d(x) - \int_{\td} f(x)^2 d\la_d(x) \bigg)\ge 0$ with strict inequality holding if and only if $f$ is not the indicator of some subset of $\td$ (upto Lebesgue null sets). Finally, observe that as $L \to \infty$ we have
\[\frac{1}{L^d}\l( \sum_k \ph_L(k)^2  \r)=  \sum_k \frac{1}{L^d}\ph\l(\frac{k}{L}\r)^2  \to \|\ph \|_2^2 .\]
This completes the proof of the proposition.
\end{proof}

% The above proposition strongly indicates (although does not prove) that the determinantal probability measures $\pf$, when $f$ is not an indicator function, is insertion and deletion tolerant. 

\section{$\pf$ determines $f$}
% Our aim here is to study the following question from \cite{Ly} and \cite{LySt}:
% \begin{question}
% \label{q01}
% Suppose that $d=1$. Note that translation and flip of $f$ yield the same measure $\pf$. Does $\pf$ determine $f$ upto translation and flip ?
% \end{question}
% We answer the question in the affirmative for ``almost all''  functions:
In this Section we provide the proofs of Theorem \ref{deter} and Proposition \ref{valdist}
\begin{proof} [Proof of Theorem \ref{deter}]
Define $\F$ to be the subset of functions $f$ of $\e$ satisfying the following conditions: 
\begin{itemize}
 \item (i) Either $\hf(k)\ne 0 \forall k \in \Z$, or $f$ is a trigonometric polynomial of degree $N$, and $\hf(k)\ne 0$ for all $|k|\le N$.  
 \item (ii) For every $n \ge 3$ (and $n \le \text{ the degree }N$ in the  case of $f$ being  a trigonometric polynomial), we have $\arg(\hf(n)) - \arg(\hf(n-1))$ does not differ from $\arg(\hf(2)) - \arg(\hf(1))$ by an integer multiple of $\pi$. Further, $\arg(\hf(2)) - 2 \arg(\hf(1))$ is not an integer multiple of $\pi$.  Here we consider $\arg$ to be a number in $(-\pi,\pi]$.  
\end{itemize}
We claim that the complement of $\F$ in $\e$, denoted by $\G:=\e \setminus \F$,  is a meagre subset of $\e$, and for $f \in \F$, we have $\pf$  determines $f$ up to the rotation and flip. A meagre subset of a topological space is a set which can be expressed as a countable union of nowhere dense sets.

To show that $\G$ is meagre, we will show that it is a subset of a countable union of nowhere dense sets. Indeed, we can write $\G$ as:
\[\G \subset \cup_{i} A_i \cup_{n \ge 3} B_{n} \cup C \]
where
\[A_i:=\{ f \in \e: \hf(i)=0 \}, \]
\[B_{n}:=\{ f \in \e: \text{ Either } \hf(k)=0 \text{ for } k=1,2,n,n-1 \text{ or } \]\[ \text{Arg}(\hf(n))-\text{Arg}(\hf(n-1)) = \arg(\hf(2)) - \arg(\hf(1)) +t\pi, \, t \text{ an integer with } |t|\le 4 \},\]
\[C:=\{ \hf(1)=0 \text{ or } \hf(2)=0 \text{ or } \arg(\hf(2))-2\arg(\hf(1)) = t\pi, \, t \text{ an integer with } |t| \le 3 \}.\]
% and
% \[C:=\{ \hf(1)=0  \text{ or } \text{Arg}(\hf(1)) \text{ is a rational multiple of } 2\pi  \}.  \]
It is not hard to see that each $A_i$ and $B_{n}$ are closed sets in $L^{\infty}(\t)$, being defined by closed conditions on finitely many co-ordinates of the Fourier expansion (observe that $|\hat{f}(n)|\le \|f\|_{\infty}$ for each $n$, hence $f \to \hat{f}(n)$ is a continuous linear functional on $L^{\infty}(\t)$). The same holds true for $C$. It is also clear that none of the $A_i$-s or $B_{n}$-s or $C$ contain any $L^{\infty}(\t)$ ball (since even small perturbations in the relevant Fourier cofficients can lead us outside these sets), showing that they are nowhere dense. All these combine to prove that $\G$ is a meagre subset of $L^{\infty}(\t)$.

Let $f$ be a function in $\F$.
We begin with the Fourier  expansion $f$: \begin{equation} \label{exp}  f(x)= \sum_{-\infty}^{\infty} a_j e^{ijx} \end{equation}
where $i$ is the imaginary unit.

We make the following observations about the expansion (\ref{exp}). First, $f$ is real valued implies 
\begin{equation}
\label{c1}
a_{-j}=\overline{a_j} \text{ for all } j.                                                                                                     
\end{equation}
Secondly, letting $f_{\xi}=f(x+\xi)=\sum_{j=-\infty}^{\infty}a_j(\xi)e^{ijx}$ where $\xi \in (-\pi,\pi]$ and the addition is in $\t$, we have
\begin{equation}
 \label{s1}
a_j(\xi)=a_j e^{ij\xi}.
\end{equation}
Finally, setting $\wt{f}(x)=f(-x)=\sum_{j=-\infty}^{\infty}\wt{a}_j e^{ijx}$, we have 
\begin{equation}
 \label{s2}
\wt{a}_j=\overline{ a_j }.
\end{equation}
We want to recover the coefficients $a_j$ (up to the symmetries (\ref{c1}), (\ref{s1}) and (\ref{s2}))  from the measure $\pf$. We observe that the class of functions $\F$ is preserved under these symmetries. In particular, a trigonometric polynomial remains a trigonometric polynomial of the same degree and the same regularity property as demanded in the definition of the class $\F$.

In what follows, we will make certain choices at several steps, which will ``spend'' these symmetries. For instance, since we need to determine $f$ only up to a translation, therefore we can use the symmetry  (\ref{s1}) to make a choice of $\xi$ and fix the argument of a particular coefficient (provided it is non-zero). In the subsequent analysis, we use this with the coefficient $a_1$ and assume that it is positive real. Similarly, the flip symmetry (\ref{s2}) can be used to fix the value of the imaginary part of a particular (non-real) coefficient to be positive. In the argument that follows, we use this with the coefficient $a_2$.

To recover $f$, we begin by observing that $a_0=\ro_1$, and is therefore determined by $\pf$. 

Further, \begin{equation} \ro_2(0,n) =  \l| \begin{array} {cc} a_0 & a_n \\ a_{-n} & a_0 \end{array} \r| = a_0^2  - |a_n|^2. \end{equation}
This implies that $\pf$ determines $|a_n|$ for all $n$.

Recall that $a_1 \ne 0$ for $f \in \F$ (unless $f$ is a constant function; this follows from condition (i) defining $\F$). Using the symmetry (\ref{s1}), we  choose $a_1$ to be a positive real number, equal to its absolute value which is determined by $\pf$. If $f \in \F$ is a trigonometric polynomial of degree 1, then we are done. Else, we proceed as follows.

For any positive integer $n$, we have 
\begin{equation}
 \ro_3(0,1,n)= \l| \begin{array} {ccc}  a_0 & a_ 1 & a_n \\ a_{-1} & a_0 & a_{n-1} \\ a_{-n} & a_{-(n-1)} & a_0 \end{array} \r|
\end{equation}   
Expanding the right hand side along the first row, we have 
\[ \ro_3(0,1,n)= a_0 \l| \begin{array} {cc} a_0 & a_{n-1} \\ a_{-(n-1)} & a_0 \end{array} \r| - a_1 \l| \begin{array} {cc} a_{-1} & a_{n-1} \\ a_{-n} & a_0 \end{array} \r| + a_n \l| \begin{array} {cc} a_{-1} & a_0 \\ a_{-n} & a_{-(n-1)} \end{array} \r|. \]
Expanding the $2 \times 2$ determinants, we can simplify the above equation to \begin{equation} \label{c_2} \ro_3(0,1,n)=2a_1 \Re(a_n \overline{a_{n-1}}) + g(a_0,a_1,|a_{n-1}|,|a_n|), \end{equation}
where $g(w,x,y,z)$ is a polynomial in four complex variables. Since all quantities in (\ref{c_2}) except $\Re(a_n \overline{a_{n-1}})$ are known and $a_1>0$, we deduce that $\pf$ determines $\Re(a_n \overline{a_{n-1}})$ for all integers $n$. 

For $n=2$, we have $a_{n-1}=a_1$, and this implies that $\pf$ in fact determines $\Re(a_2)$. Since $|a_{2}|$ is also known, this implies that $\pf$ determines $|\Im(a_{2})|$, which is non-zero because of the assumption $\arg(\hf(2))-2\arg(\hf(1))$ is not an integer multiple of $\pi$, $\arg(\hf(1))$ being 0 because $\hf(1)$ is real. Using the symmetry (\ref{s2}), we choose $a_{2}$ such that $\Im(a_{2})>0$. 

We have now spent all the symmetries present in the problem, and our goal is to show that all the other $a_n$-s ($n \ge 0$) are determined exactly by $\pf$. $a_n$ for $n <0$ can then be found using the symmetry (\ref{c1}).

To this end, we apply induction. Suppose $n \ge 3$ and we know the values of $a_k, 0\le k \le n-1$. Since $f \in \F$, either $a_n=0$ or all such $a_k$ are non-zero. Computing $\ro_3(0,2,n)$ along similar lines to $\ro_3(0,1,n)$ we obtain 
\begin{equation} \label{c_3} \ro_3(0,2,n)=2 \Re(a_n \overline{a_2a_{n-2}}) + g(a_0,a_2,|a_{n-2}|,|a_n|), \end{equation}
where $g$ is as in (\ref{c_2}). 

If $|a_n|=0$, then we deduce that $f$ is a trigonometric polynomial of degree $n-1$, because $f \in \F$ (recall condition (i) defining $\F$).  In that case, we have already determined $f$. Else, $a_n\ne 0$, and we proceed as follows.

Since we already know $|a_n|$, we need only to determine $\text{Arg}(a_n)$. For any complex number $z \ne 0$, $|z|$ and $\Re(z)$ determines $\arg(z)$ (considered as a number in $(-\pi,\pi]$) up to sign. Hence, if we know $\Re(z\overline{z_1})$ and $\Re(z\overline{z_2})$ for two non-zero complex numbers $z_1$ and $z_2$ such that $\arg(z_1)$ does not differ from $\arg(z_2)$ by an integer multiple of $\pi$, then this data would be sufficient to determine $\arg(z)$. To be more elaborate, if $\theta,\phi_1,\phi_2$ are the arguments of $z,z_1$ and $z_2$ respectively, then the given data determines $\theta - \phi_1$ and $\theta - \phi_2$ upto signs (via the cosines of these quantities). Unless $\phi_1 - \phi_2$ is an integer multiple of $\pi$, this is enough to determine $\theta$. We apply  this to the situation we have in our hands, with $z=a_n$, $z_1=a_{n-1}$ and $z_2=a_2a_{n-2}$. None of them is $0$ by condition (i) defining $\F$, and the condition on the difference of arguments of $z_1$ and $z_2$ follows from 
condition (ii) defining $\F$. This enables us to determine $\arg(a_n)$, and hence $a_n$.

This completes the proof.
\end{proof}

\begin{remark}
The argument used in proof of Theorem \ref{deter} can also be used to recover the function $f$ if all its Fourier coefficients are real. E.g., if $f$ is the indicator function of an interval $A$ in $(-\pi,\pi]$, then we can ``rotate'' $f$ (symmetry \ref{s1}) so that 0 is in the centre of the interval $A$. Then all the Fourier coefficients of $f$ are real, and we can identify the interval $A$. This argument will also work for any set $A$ which has a point of reflectional symmetry when looked upon as a subset of $\t$.   
\end{remark}

\begin{proof} [Proof of Proposition \ref{valdist}]
Recall that the harmonic mean of a function $f:\td \to \R$ is defined as
\[\hm(f)=\bigg(\int_{\td} \frac{d\la_d(x)}{f(x)} \bigg)^{-1},\] where $\la_d$ is the Lebesgue measure on $\td$. 
The fact that we know the distribution $\pf$ implies that we know the distribution $\P^{tf}$ for any $0<t<1$, e.g. by performing an independent site percolation with survival probability $t$ on $\pf$. By taking complements, that is by considering the point process of the excluded points  (see \cite{LySt} for more details), this implies that we know the distribution $\P^{1-tf}$. But this enables us to recover the harmonic mean $\hm(1-tf)$ of the function $1-tf$ by the formula \[\hm(1-tf)= \text{Sup} \{ p \in [0,1]: \mu_p \preccurlyeq_{\text{f}} \P^{1-tf} \}.\] 
Here $\mu_p$ is the standard site percolation on $\Z$ with survival probability $p$, and $\mu_p \preccurlyeq_{\text{f}} \P^{1-tf}$ means that $\P^{1-tf}$ is uniformly insertion tolerant at level $p$, that is, $\P^{1-tf}[\o(0)=1|\ou]\ge p$ a.s. in $\ou$. For details, we refer to Definition 5.15 and Theorem 5.16 in \cite{LySt}. But \[ \hm(1-tf)^{-1} = \int_{\td}\frac{d\la_d(x)}{1-tf(x)}.\] By expanding the integral on the right as a power series in $t$, we can recover $\int_{\td}f^k(x) \, d\la_d(x)$ for each $k \ge 1$. But we have \[\int_{\td}f^k(x) \, d\la_d(x)= k \int_{0}^{ 1} \xi^{k-1} \nu_f(\xi) d\el(\xi),\] where $\el$ is the Lebesgue measure on $\R$, and $\nu_f$ is the value distribution of $f$, given by \[\nu_f(\xi)=\la_d(\{x \in \td: f(x) \ge \xi\}).\] Thus we have all the moments of the measure $\nu_f(\xi)d\el(\xi)$. Since $\nu_f(\xi)d\el(\xi)$ is a compactly supported measure on the interval $[0,1]$, it is uniquely determined by its moments (this follows, e.g., from the Weierstrass approximation theorem on the unit interval).
% its Fourier transform is finite everywhere, and the moments enable us to compute the Fourier transform.
%  The Fourier transform determines the measure, and hence $\nu_f$, uniquely. 
% , and hence the moment generating function of $\nu_f$. It only remains to show that the moment generating function of $\nu_f$ is finite in an interval. But for $0<t<1$, $e^{tf(x)}\le \frac{1}{1-tf(x)}$, so $HM(1-tf)^{-1}$ being finite implies that the moment generating function of $\nu_f$ is finite for $0< t < 1$ 
This enables us to recover the value distribution of $f$. 
\end{proof}

\section{Negative association for determinantal point processes}
In this section we take up the proof of Theorem \ref{cna}, and remark how it connects with more general notions of negative association. In doing so, we will make use of a discretization approach due to Goldman \cite{Go}, which Lyons in \cite{Ly1} calls the transference principle. To this end, we quote below Theorem 3.4 and Lemma 3.5 from \cite{Ly1}; they correspond to Proposition 12 and Lemma 16 in Goldman's original paper \cite{Go}.

\begin{theorem}
 \label{transfer1}
Let $(E,\mu)$ and $(F,\nu)$ be two Radon measure spaces on locally compact Polish sets. Let $(A_i)_{i=1}^n$ and $(B_i)_{i=1}^n$ be pairwise disjoint Borel subsets of $E$ and $F$ respectively. Let $\la_k \in [0,1]$ with $\sum_k \la_k <\infty$. Let $(\phi_k)$ and $(\psi_k)$ be orthonormal vectors in $L^2(E,\mu)$ and $L^2(F,\nu)$ respectively. Let $K:=\sum_k \la_k \phi_k \otimes \ol{\phi_k}$ and $L:=\sum_k \la_k \psi_k \otimes \ol{\psi_k}$. If $\langle  1_{A_i} \psi_j , \psi_k \rangle = \langle  1_{B_i} \psi_j , \psi_k \rangle$ for all possible $i,j,k$, then the $\P^K$ distribution of $(N(A_1),\cdots,N(A_n))$ is equal to the $\P^L$ distribution of $(N(B_1),\cdots,N(B_n))$. Here $\P^K$ and $\P^L$ are the probability measures corresponding to the determinantal point processes induced by $K$ and $L$ on $E$ and $F$ respectively.
\end{theorem}

\begin{proposition}
\label{transfer2}
Let $\mu$ be a Radon measure on a locally compact Polish space $E$. Let $(A_i)_{i=1}^N$ be pairwise disjoint Borel subsets of $E$. Let $\phi_k \in L^2(E,\mu)$ for $k \ge 1$ be orthonormal. Then there exists a denumerable set $F$, pairwise disjoint subsets $(B_i)_{i=1}^N$ of $F$, and orthonormal vectors $v_k \in \ell^2(F)$ such that $\langle 1_{A_i} \phi_j, \phi_k \rangle = \langle  1_{B_i} v_j , v_k \rangle $ for all possible values of $i,j,k$.
\end{proposition}

We now apply these results in order to establish Theorem $\ref{cna}$.

\begin{proof}[Proof of Theorem \ref{cna}]
First, we show that it suffices to consider the case where each $A_i$ is precompact. Because, we can consider an increasing sequence of compact sets $M_r$ such that $\cup_{r=1}^{\infty}M_r=E$. Now consider $A_i^r= A_i \cap M_r$. If we have the statement of the theorem for each $A_i^r$, then we can let $r \uparrow \infty$ and observe that $1_{N(A_i^r) \le m_i} \downarrow 1_{N(A_i) \le m_i}$ to deduce the theorem in the general case.

Henceforth, we assume that each $A_i$ is indeed precompact.

Denote $A=\cup_{i=1}^n A_i$. Denote by $K_A$ the compression of the integral operator $K$ to the set $A$. In other words, $K_A$ is an integral operator with the kernel $K_A(x,y)=1_A(x)K(x,y)1_A(y)$. So $K_A$ is a trace class operator. Let $\{\phi_i\}_{i=1}^{\infty}$ be an orthonormal eigenbasis for the integral operator $K_A$, with the eigenvalue $\la_k$  corresponding to the eigenfunction $\phi_k$. Thus, the integral kernel $K_A(x,y)$ has the  expansion \[K_A(x,y)=\sum_k \la_k \phi_k(x) \otimes \ol{\phi_k(y)}.\] Here each $\la_k \in [0,1]$ because $K$ must be a non-negative contraction.

By Proposition \ref{transfer2}, we have a countable set $F$,  orthonormal vectors $\{v_i\}_{i=1}^{\infty}$ in $\ell^2(F)$ and pairwise disjoint subsets $(B_i)_{i=1}^{n}$ of $F$  such that $\langle 1_{A_i} \phi_j, \phi_k \rangle = \langle  1_{B_i} v_j , v_k \rangle $ for all possible values of $i,j,k$. Define $L:\ell^2(F) \to \ell^2(F)$ to be the integral operator with the integral kernel \[L(x,y)=\sum_k \la_k v_k(x) \otimes \ol{v_k(y)}.\]  Thus, $L$ is a non-negative contraction on $\ell^2(F)$, and therefore, defines a determinantal point process on $F$. By Theorem \ref{transfer1}, this implies that the laws of $(N(A_1),\cdots, N(A_n))$ and  $(N(B_1),\cdots, N(B_n))$ are identical. 

But $\P^L$ is the law of a determinantal point process on a countable space $F$, hence by Theorem 6.5 in \cite{Ly}, $\P^L$ has negative associations. Namely,  \[\E_{\P^L}[f_1 \cdots f_n] \le \E_{\P^L}[f_1] \cdots \E_{\P^L}[f_n]\] for any collection $f_1,\cdots,f_n$ of non-negative functions (on point configurations on $F$)  that are measurable with respect to disjoint subsets of $F$, and that are either all increasing or all decreasing.

Setting $f_i=1_{N(B_i) \le m_i}$ in the above, we get \[\P^L \l( \bigcap_{i=1}^n (N(B_i)\le m_i)  \r)  \le \prod_{i=1}^n \P^L(N(B_i)\le m_i). \] But $(N(A_1),\cdots, N(A_n))$ has the same law as  $(N(B_1),\cdots, N(B_n))$, so this implies \[\P^K \l( \bigcap_{i=1}^n (N(A_i)\le m_i)  \r)  \le \prod_{i=1}^n \P^K(N(A_i)\le m_i) , \] as desired.
\end{proof}

We observe the following corollary of the proof of Theorem \ref{cna}:

\begin{corollary}
 \label{cnacor}
 Let $f:\Z_{\ge 0}^m \to \R$ and $g:\Z_{\ge 0}^n \to \R$ be two functions which are co-ordinate wise non-decreasing. Let $\{A_i\}_{i=1}^m$ and $\{B_j\}_{j=1}^n$ be disjoint bounded Borel sets. Let $N(A)$ denote the number of points of the determinantal point process $\Pi$ with a continuous kernel. Let $X_f=f(N(A_1),\cdots,N(A_m))$ and $X_g=g(N(B_1),\cdots,N(B_n))$.  Then we have \[\E \l[ X_f X_g \r] \le \E \l[ X_f \r] \E \l[ X_g \r] .\]
\end{corollary}

\begin{proof}
 The proof is on similar lines to Theorem \ref{cna}, using the transference principle and utilizing the analogous negative correlation inequality for discrete determinantal processes (as a special case of Theorem 6.5 in \cite{Ly})  .
\end{proof}

\textbf{Connection of Theorem \ref{cna} to more general notions of negative association}

It can be shown that Theorem \ref{cna} implies much more general notions of negative association. In the note \cite{Yo}, D. Yogeshwaran has shown that Corollary \ref{cnacor} implies negative correlation of increasing (or decreasing) functionals of the point process which are supported on disjoint subsets of $\R^d$, provided that these functionals are $\Pi$-continuous (i.e., their set of discontinuities has $\Pi$-measure 0).

\textbf{Acknowledgements.} The author would like to thank Russell Lyons, Vadim Gorin, Manjunath Krishnapur and Yuval Peres for stimulating discussions. The author is very grateful to the anonymous referee for an extremely careful reading of the paper and his/her valuable suggestions, in particular for the removal of a continuity assumption on the kernel in Theorem \ref{rigcomp}.

\end{document}